\documentclass[a4paper,reqno]{amsart}
\usepackage{amsmath, amsthm, amsfonts, amssymb}
\usepackage{graphicx}

\newcommand{\mR}{{\mathbb R}}
\newcommand{\mP}{{\mathbb P}}

\newcommand{\II}{\mathbf{1}}

\newtheorem{theorem}{Theorem}

\newtheorem{example}{Example}
\newtheorem{remark}{Remark}[section]
\newtheorem{definition}{Definition}

\newtheorem{proposition}[theorem]{Proposition}

\begin{document}
\title[Stochastic Orderings and generalized Stochastic Precedence]
{Relations Between Stochastic Orderings and generalized Stochastic Precedence}

\author[]{  Emilio De Santis}
\email{desantis@mat.uniroma1.it}

\author[]{Fabio Fantozzi}
\email{fantozzi@mat.uniroma1.it}

\author[]{Fabio Spizzichino\\Department of Mathematics - University La Sapienza, Rome}
\email{spizzich@mat.uniroma1.it}

\begin{abstract}

The concept of \textit{stochastic precedence} between two real-valued random
variables has often emerged in different applied frameworks. In this paper we consider a slightly
more general, and completely natural, concept of stochastic precedence and
analyze its relations with the notions of stochastic ordering. Such a study
leads us to introducing some special classes of bivariate copulas.
Motivations for our study can arise from different fields.
In particular we consider the frame of Target-Based Approach in decisions under risk.
This approach has been mainly developed under the
assumption of stochastic independence between ``Prospects'' and ``Targets''.
Our analysis concerns the case of stochastic dependence.

\bigskip


\end{abstract}

\date{\today} \maketitle

\section{Introduction.} \label{sec:Intro}

Let $X_1,X_2$ be two real random variables defined on a same probability
space $(\Omega,\mathcal{F},\mathbb{P})$. We will denote by $F$ the
joint distribution function and by $G_1$, $G_2$ their marginal
distribution functions, respectively. For the sake of notational simplicity,
we will initially concentrate our attention on the case when $G_1,G_2$
belong to the class $\mathcal{G}$ of all the probability distribution
functions on the real line, that are continuous and strictly increasing in the
domain where they are positive and smaller than one. As we shall see later,
we can also consider more general cases, but the present restriction allows us to simplify the
formulation and the proofs of our results. In order to account for some cases of interest
with $\mathbb{P}(X_1=X_2)>0$, we will not assume that the distribution function $F$ is absolutely
continuous.

The random variable $X_1$ is said to \textit{stochastically precede} $X_2$ if
$\mathbb{P}(X_1\leq X_2)\geq1/2$, written $X_1\preceq_{sp}X_2$. The interest of
this concept for applications has been pointed out several times in the
literature (see in particular \cite{AKS}, \cite{BolSin} and \cite{NavRub}).
We recall the reader's attention on the fact that stochastic precedence
does not define a stochastic order in that, for instance, it is not transitive.
However it can be considered in some cases as an interesting condition, possibly alternative
to the usual \textit{stochastic ordering} $X_1\preceq_{st}X_2$, defined by the inequality $G_1(t)\geq
G_2(t),\,\forall t\in\mathbb{R}$, see \cite{ShSh}.

When $X_1,X_2$ are independent the implication $X_1\preceq_{st}X_2\Rightarrow X_1\preceq_{sp}X_2$
holds (see \cite{AKS}). It is also easy to find several other examples of bivariate probability models
where the same implication holds. For instance the condition $X_1\preceq_{st}X_2$ even entails
$\mathbb{P}(X_1\leq X_2)=1$ when $X_1,X_2$ are \textit{comonotonic} (see e.g.
\cite{nelsen2006introduction}), i.e. when $\mP(X_2=G_2^{-1}(G_1(X_1)))$.
On the other hand, cases of stochastic dependence can be found
where the implication $X_1\preceq_{st}X_2\Rightarrow X_1\preceq_{sp}X_2$ fails.
A couple of examples will be presented in Section \ref{sec3}. See also Proposition \ref{BCgamma}.
On the other hand the frame of words' occurrences produces, in a natural way, examples in the same direction, see e.g. \cite{DS2012}.

In this paper we replace the notion $X_1\preceq_{sp}X_2$ with the generalized concept defined as follows
\begin{definition}
For given $\gamma\in[0,1]$, we say that $X_1$ stochastically precedes
$X_2$ at level $\gamma$ if $\mathbb{P}(X_1\leq X_2)\geq\gamma$. This will be written
$X_1\preceq_{sp}^{(\gamma)}X_2$.
\end{definition}

\smallskip

Let $\mathcal{C}$ denote the class of all bivariate copulas (see e.g.
\cite{joe1997,nelsen2006introduction}).
Several arguments along the paper, we be based on the concept of bivariate copula and
the class of all bivariate copulas will be denoted by $\mathcal{C}$.
We say that the pair of random variables $X_1,X_2$, with distributions $G_1,G_2$,
respectively, admits $C\in\mathcal{C}$ as its \textit{connecting} copula
whenever its joint distribution function is given by

\begin{equation}\label{connecting}
F(x_1,x_2)=C(G_1(x_1),G_2(x_2)).
\end{equation}

\medskip

It is well known (see e.g. \cite{nelsen2006introduction}) that the connecting copula is unique
when $G_1$ and $G_2$ are continuous. We will use the notation
\begin{equation}
\label{set}A:=\{(x_1,x_2)\in\mathbb{R}^{2}\,:\,x_1\leq x_2\},
\end{equation}
so that we write
\begin{equation}
\label{eta2}{\mathbb{P}}(X_1\leq X_2)=\int_{A}\,dF(x_1,x_2)=\int_{\mathbb{R}^2}
\mathbf{1}_{A}(x_1,x_2)\,dF(x_1,x_2).
\end{equation}
For given $G_1,G_2\in\mathcal{G}$ and $C\in\mathcal{C}$ we also set
\begin{equation}\label{eta}
\eta(C,G_1,G_2):=\mathbb{P}(X_1\leq X_2),
\end{equation}
where $X_1$ and $X_2$ are random variables with distributions $G_1,G_2$
respectively, and connecting copula $C$. Thus the condition
$X_1\preceq_{sp}^{(\gamma)}X_2$ can also be written $\eta(C,G_1,G_2)\geq\gamma$.

Suppose now that $X_1,X_2$ satisfy the condition $X_1\preceq_{st}X_2$.
As a main purpose of this paper we give a lower bound for the probability $\mathbb{P}(X_1\leq X_2)$
in terms of the stochastic dependence between $X_1$ and $X_2$ or, more precisely, in terms of conditions
on the integral $\int_{A\cap[0,1]^2}dC$. More specifically we will analyze different
aspects of the special classes of bivariate copulas, defined as follows.

\begin{definition}\label{def2}
For $\gamma\in[0,1]$, we denote by $\mathcal{L}_{\gamma}
$ the class of all copulas $C\in\mathcal{C}$ such that
\begin{equation}\label{BasicImplication}
\eta(C,G_1,G_2)\geq\gamma
\end{equation}
for all $G_1,G_2\in\mathcal{G}$ with $G_1\preceq_{st} G_2$.
\end{definition}

\medskip

Concerning the role of the concept of copula in our study, we point out the following simple facts.
Consider the random variables $X_1'=\phi(X_1)$ and $X_2'=\phi(X_2)$ where $\phi:\mathbb{R}\rightarrow\mathbb{R}$
is a strictly increasing function. Thus $X_1'\preceq_{st}X_2'$ if and only if $X_1\preceq_{st}X_2$
and $X_1'\preceq_{sp}^{(\gamma)}X_2'$ if and only if $X_1\preceq_{sp}^{(\gamma)}X_2$.
At the same time the pair $X_1',X_2'$ also admits the same connecting copula $C$.

The arguments treated in this paper can reveal of interest in the frame of different applied fields.
Motivations for this study, in particular, had arisen for us from the following two fields:
\begin{itemize}
  \item[i)] the Target-Based Approach in utility theory;
  \item[ii)] comparisons among waiting times to occurrences of \emph{words} in random sequences of \emph{letters}
  from an alphabet.
\end{itemize}

Further applications can arise e.g. in the fields of reliability and in the comparison of pool
obtained by two opposite coalitions.

More precisely the structure of the paper is as follows. In Section
\ref{sec2}, we analyze the main aspects of the class
$\mathcal{L}_{\gamma}$ and present a related characterization. Some
further basic properties will be detailed in Section \ref{sec3},
where a few examples will be also presented. Finally, in Section \ref{sec4}, we will briefly review
Target-Based utilities, pointing out the relations with our work, in the case of stochastic dependence between
targets and prospects. Connections with the field of times to words' occurrences will be discussed in a subsequent note.

\section{A characterization of the class $\mathcal{L}_{\gamma}$.}\label{sec2}

This Section will be devoted to providing a characterization of the
class $\mathcal{L}_{\gamma}$ (see Theorem \ref{Lgamma} and \ref{continuo}) along with related discussions.
We start by detailing a few basic properties of the quantities $\eta(C,G_1,G_2)$, for
$G_1,G_2\in\mathcal{G}$ and $C\in\mathcal{C}$.
In view of the condition $G_1,G_2\in\mathcal{G}$ we can use the change of variables
$u=G_1(x_1)$, $v=G_2(x_2)$. Thus we can rewrite the integral in \eqref{eta2} according to the following

\begin{proposition}\label{eta_cfg}
For given $G_1,G_2\in$ $\mathcal{G}$ and $C\in\mathcal{C}$, one has

\begin{equation}\label{formulautile}
\eta(C,G_1,G_2)=\int_{[0,1]^2}
\mathbf{1}_A(G_1^{-1}(u),G_2^{-1}(v))\,dC(u,v).
\end{equation}
\end{proposition}

\medskip

The use of the next Proposition is two-fold: it will be useful both for characterizing the class
$\mathcal{L}_\gamma $ and establishing lower and upper bounds on the quantity $\eta(C,G_1,G_2)$.

\begin{proposition}\label{eta_ineq}
Let $G_1,G_1',G_2 ,G_2'\in\mathcal{G}$. Then
\begin{align*}
G_2\preceq_{st}G_2'&\Rightarrow\eta(C,G_1,G_2)\leq\eta(C,G_1,G_2');\\
G_1\preceq_{st}G_1'&\Rightarrow\eta(C,G_1,G_2)\geq\eta(C,G_1',G_2).
\end{align*}
\end{proposition}

\begin{proof}
We prove only the first relation of Proposition \ref{eta_ineq}, since the
proof for the second one is analogous. By hypothesis, and since $G_1,G_2'\in\mathcal{G}$ for each $x\in(0,1)$,
one has
$$
G_2^{-1}(x)\leq G_2'^{-1}(x).
$$
Therefore
$$
(G_1^{-1}(x),G_2^{-1}(x))\in A\Rightarrow(G_1^{-1}(x),G_2'^{-1}(x))\in A\,.
$$
Hence, the proof can be concluded by recalling \eqref{formulautile}.
\end{proof}

\medskip

From Proposition \ref{eta_ineq}, in particular we get
\begin{equation}\label{Disug1}
\eta(C,G,G)\leq\eta(C,G',G)\qquad\mbox{ and }\qquad
\eta(C,G,G)\leq\eta(C,G,G'')\,,
\end{equation}
for any choice of $G,G',G''\in\mathcal{G}$ such that $G'\preceq_{st}G\preceq_{st}G''$.

\medskip

A basic fact in the analysis of the classes $\mathcal{L}_{\gamma}$ is that the
quantities of the form $\eta(C,G,G)$ only depend on the copula $C$. More
formally we state the following result.

\begin{proposition}\label{foranypair}
For any pair of distribution functions
$G^{\prime},G^{\prime\prime}\in\mathcal{G}$, one has
\begin{align}
\eta(C,G^{\prime},G^{\prime})=\eta(C,G^{\prime\prime},G^{\prime\prime})\,,\label{etaCGG}
\end{align}
\end{proposition}

\begin{proof}
Recalling \eqref{formulautile} one obtains
\begin{equation}\label{2formulautile}
\int_{\lbrack0,1]^{2}}\mathbf{1}_{A}(G^{\prime-1}(u),G^{\prime-1}(v))\,dC(u,v)=
\int_{[0,1]^{2}}\mathbf{1}_{A}(G^{\prime\prime-1}(u),G^{\prime\prime-1}(v))\,dC(u,v)
\end{equation}
because $\mathbf{1}_{A}(G^{\prime-1}(u),G^{\prime-1}(v))=
\mathbf{1}_{A}(G^{\prime\prime-1}(u),G^{\prime\prime-1}(v))=\mathbf{1}_{A}(u,v)$,
so equality in \eqref{etaCGG} is proved.
\end{proof}

\medskip

As a consequence of Proposition \ref{foranypair} we can introduce the symbol
\begin{align}\label{etaC}
    \eta(C):=\eta(C,G,G),
\end{align}
and, by letting $G_1=G_2=G$ in \eqref{formulautile}, write
\begin{equation}\label{etapuro}
\eta(C)=\int_{A\cap[0,1]^2}dC
\end{equation}
for $G\in\mathcal{G}$. From Proposition \ref{eta_ineq} and from the inequalities \eqref{Disug1}, we obtain

\begin{proposition}\label{etainf}
For $G_1,G_2\in\mathcal{G}$ the following implication holds
$$
G_1\preceq_{st}G_2\Rightarrow\eta(C)\leq\eta(C,G_1,G_2).
$$
\end{proposition}

\medskip

We then see that the quantity $\eta(C)$ characterizes the
class $\mathcal{L}_{\gamma},\,0\leq\gamma\leq1$, in fact we can state the following

\begin{theorem}\label{Lgamma}
$C\in\mathcal{L}_{\gamma}$ if and only if $\,\eta(C)\geq\gamma$.
\end{theorem}

\medskip

We thus have
\begin{equation}\label{LgammaC}
\mathcal{L}_\gamma =\{C\in\mathcal{C}:\eta(C)\geq\gamma\}
\end{equation}
and we can also write
\begin{equation}\label{Lgamma2}
\eta(C)=\inf_{G_1,G_2\in \mathcal{G}}\{\eta(C,G_1,G_2):G_1\preceq_{st}G_2\}.
\end{equation}
In other words the infimum in formula \eqref{Lgamma2} is a minimum and it is attained when $G_1=G_2$.
We notice furthermore that the definition of $\eta(C,G_1,G_2)$ can be extended to the case when
$G_1,G_2\in D(\mathbb{R})$, the space of distribution functions on $\mathbb{R}$. The class
$\mathcal{G}$ has however a special role in the present setting, as it is shown in the following result.

\begin{theorem}
\label{continuo} Let $C\in\mathcal{C} $, $G,H\in D(\mathbb{R})$
with $G\preceq_{st} H$, then $\eta(C,G,H)\geq \eta(C)$.
\end{theorem}

\begin{proof}
Consider two  sequences $(G_n: n \in\mathbb{N})$, $(H_n: n\in\mathbb{N})$ such that
$G_n, H_n\in \mathcal{G}$ and $ G_n\stackrel{w}{\to}G $, $H_n \stackrel{w}{\to}H $. Applying the
Theorem 2 in \cite{Sempi2004}, we obtain that $C(G_n ,H_n)
\stackrel{w}{\to} C(G,H) $.

Take the new sequence $(\tilde H_n: n \in\mathbb{N}) $ where
$\tilde H_n (x) := \min\{G_n(x), H_n(x)\} $. We notice
that $\tilde H_n \in\mathcal{G}$, moreover $G_n \preceq_{st}\tilde H_n$ and
$\tilde H_n \stackrel{w}{\to} H$.
This implies $C(G_n, \tilde H_n)\stackrel{w}{\to} C(G,H)$.

Now using the standard characterization of weak convergence on
separable spaces (see \cite{billingsley2009convergence} p. 67
Theorem 6.3)
\[
\limsup_{n\to\infty}\int_{B}d\tilde F_n\leq\int_{B}dF\,,
\]
for any closed set $B\in\mathbb{R}^2$, where $F=C(G,H)$ and $\tilde F_n=C(G_n, \tilde H_n)$.
Taking the closed set $A$ defined in \eqref{set} one has
\begin{equation}\label{limsup}
\eta(C)\leq\limsup_{n\to\infty}\int_A d\tilde F_n\leq\int_A dF=\eta(C,G,H).
\end{equation}
\end{proof}

\begin{remark}
Theorem \ref{continuo} shows that the minimum of $\eta(C,G,H)$, for
$G,H\in D(\mathbb{R})$, is attained at $(C,G,G)$, for any
$G\in\mathcal{G}\subset D(\mathbb{R})$. This result allows us to
replace  the class $\mathcal{G}$ with $D(\mathbb{R})$ in the
expression of $\mathcal{L}_\gamma$ given in \eqref{Lgamma2}.
We notice furthermore that one can have $\eta(C,G',G')\neq\eta(C,G'',G'')$
when $G',G''$ are in $D(\mathbb{R})$.
\end{remark}

\bigskip

Concerning the classes $\mathcal{L}_\gamma$, we also define
\begin{equation}
\mathcal{B}_{\gamma}:=\{C\in\mathcal{C}\,|\,\eta(C)=\gamma\},
\end{equation}
so that
$$
\mathcal{L}_{\gamma}=\bigcup_{\gamma'\geq\gamma}\mathcal{B}_{\gamma'}.
$$

\medskip

We now show that the classes $\mathcal{B}_{\gamma}$, $\gamma \in
[0,1]$, are all non empty. Several natural examples might be
produced on this purpose. We fix attention on
a simple example built in terms of the random variables $X_1$,
$X_2^{(\gamma)}$ defined as follows.
On the probability space $([0,1],\mathcal{B}[0,1],\lambda)$, where
$\lambda$ denotes the Lebesgue measure, we take $X_1(\omega)=\omega$, and
\begin{equation}\label{Shuffle}
X_2^{(\gamma)}(\omega)= \left\{
\begin{array}[c]{ll}
\omega+1-\gamma & \text{ if }\omega\in[0,\gamma],\\
\omega-\gamma & \text{ if }\omega\in(\gamma,1].
\end{array}
\right.
\end{equation}

As it happens for $X_1$, also the distribution of $X_2^{(\gamma)} $ is uniform in $[0,1]$ for any $\gamma\in[0,1]$
and the connecting copula of $X_1,X_2^{(\gamma)}$, that is then uniquely determined, will be denoted by $C_\gamma$.

\begin{proposition}\label{BCgamma}
For any $\gamma\in(0,1]$, one has
\begin{itemize}
\item[(i)]
$C_\gamma\in\mathcal{B}_\gamma$.
\item[(ii)]
$C_{\gamma}(u,v)=\min\{u,v,\max\{u-\gamma,0\}+\max\{v+\gamma-1,0\}\}$.
\label{copulaCgamma1}
\end{itemize}
\end{proposition}

\begin{proof}

(i) First we notice that ${\mathbb{P}}(X_1\leq X_2^{(\gamma)})=\gamma$. In fact
$$
\mathbb{P}(X_1\leq X_2^{(\gamma)})={\mathbb{P}}(X_1\leq X_1+1-\gamma,X_1\leq \gamma)+
\mathbb{P}(X_1\leq X_1-\gamma,X_1>\gamma)=\gamma.
$$
Whence, $\eta(C_{\gamma})=\mathbb{P}(X_1\leq X_2^{(\gamma)})=\gamma$, since
both the distributions of $X_1,X_2^{(\gamma)}$ belong to $\mathcal{G}$.

(ii) For $x_1,x_2\in [0,1]$ we can write
\begin{align*}
F_{X_1,X_2^{(\gamma)}}(x_1,x_2)
& :=\mathbb{P}(X_1\leq x_1,X_2^{(\gamma)}\leq x_2)\\
& =\mathbb{P}(X_1\leq x_1,X_1+1-\gamma\leq x_2,\,X_1\leq\gamma)\\
& \,+\mathbb{P}(X_1\leq x_1,X_1\leq x_2+\gamma,\,X_1>\gamma)\\
& =\mathbb{P}(X_1\leq\min\{x_1,x_2+\gamma-1,\gamma\})+\mathbb{P}(\gamma<X_1\leq\min\{x_1,x_2+\gamma\})\\
& =\max\{\min\{x_1,x_2+\gamma-1,\gamma\},0\}+\max\{\min\{x_1,x_2+\gamma\}-\gamma,0\}\\
& =\min\{x_1,x_2,\max\{x_1-\gamma,0\}+\max\{x_2+\gamma-1,0\}\}\,.
\end{align*}
Since both the marginal distributions of $X_1$ and $X_2^{(\gamma)}$
are uniform, it follows that
$$
C_{\gamma}(u,v)=\min\{u,v,\max\{u-\gamma,0\}+\max\{v+\gamma-1,0\}\}.
$$
\end{proof}

The copulas $C_\gamma$ have also been considered for different purposes in the literature, see e.g.
\cite{Nelsen2007} and \cite{SibSto2008}.
We point out that the identity $\eta(C_\gamma)=\gamma$ (for $\gamma\in(0,1]$) could also have been obtained
directly from formula \eqref{etapuro}. In this special case the computation of $\mP(X_1\leq X_2)$ is
however straightforward.

As an immediate consequence of Proposition \ref{BCgamma} we have that $\mathcal{L}_{\gamma'}$
is strictly contained in $\mathcal{L}_{\gamma}$ for any $0\leq\gamma<\gamma^{\prime}\leq1$.
We notice furthermore that $\mathcal{L}_0=\mathcal{C}$ and
$\mathcal{L}_1=\{C\in\mathcal{C} : \int_{A\cap [0,1]^2}dC=1\}\neq \emptyset$.


Graphs of $C_\gamma$ for different values of $\gamma$ are provided in Figure \ref{fig:Cgamma}.

\begin{figure}[ht]
\centering
  \includegraphics[width=0.3\textwidth,keepaspectratio]{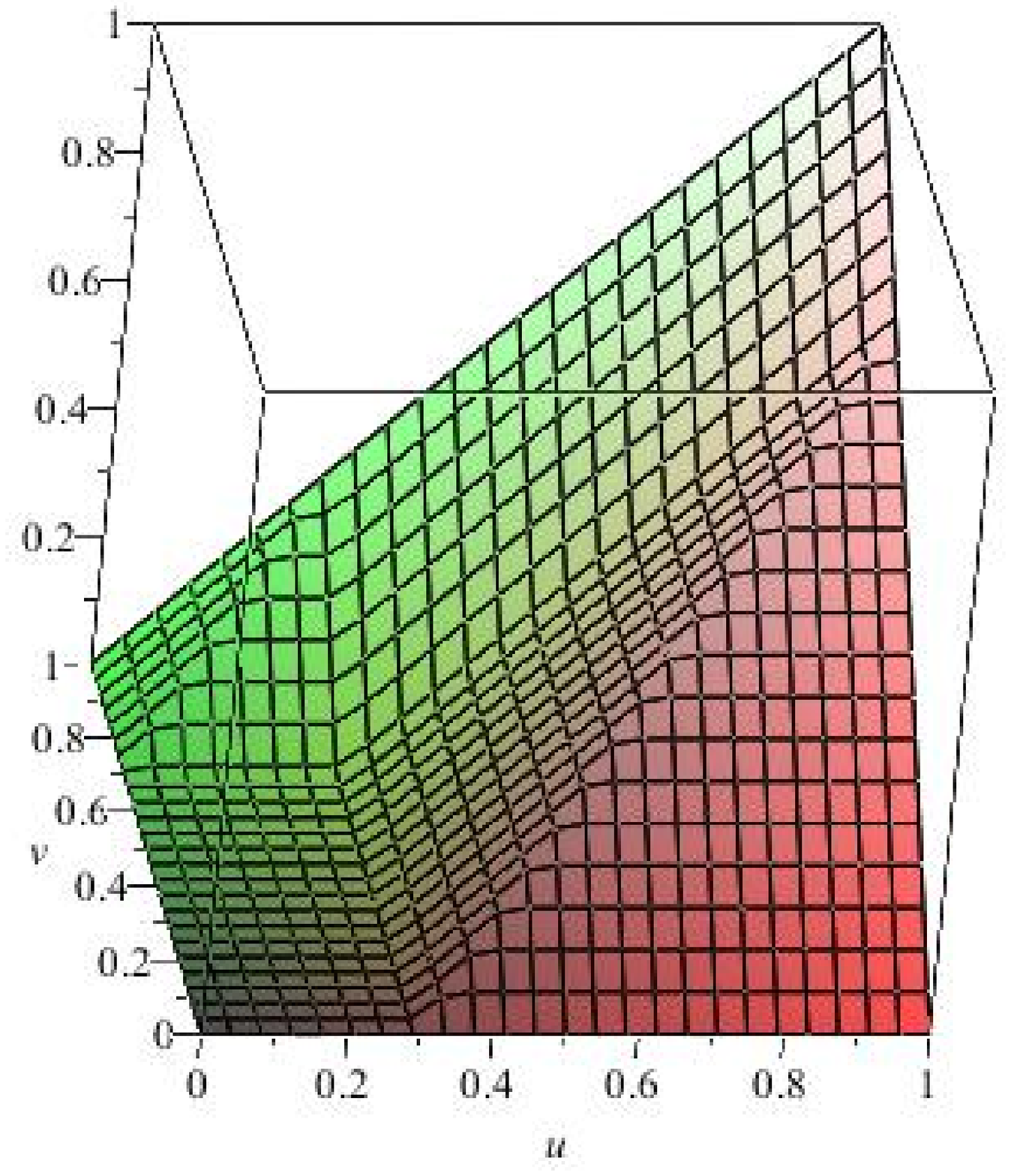}
  \hspace{1mm}
  \includegraphics[width=0.3\textwidth,keepaspectratio]{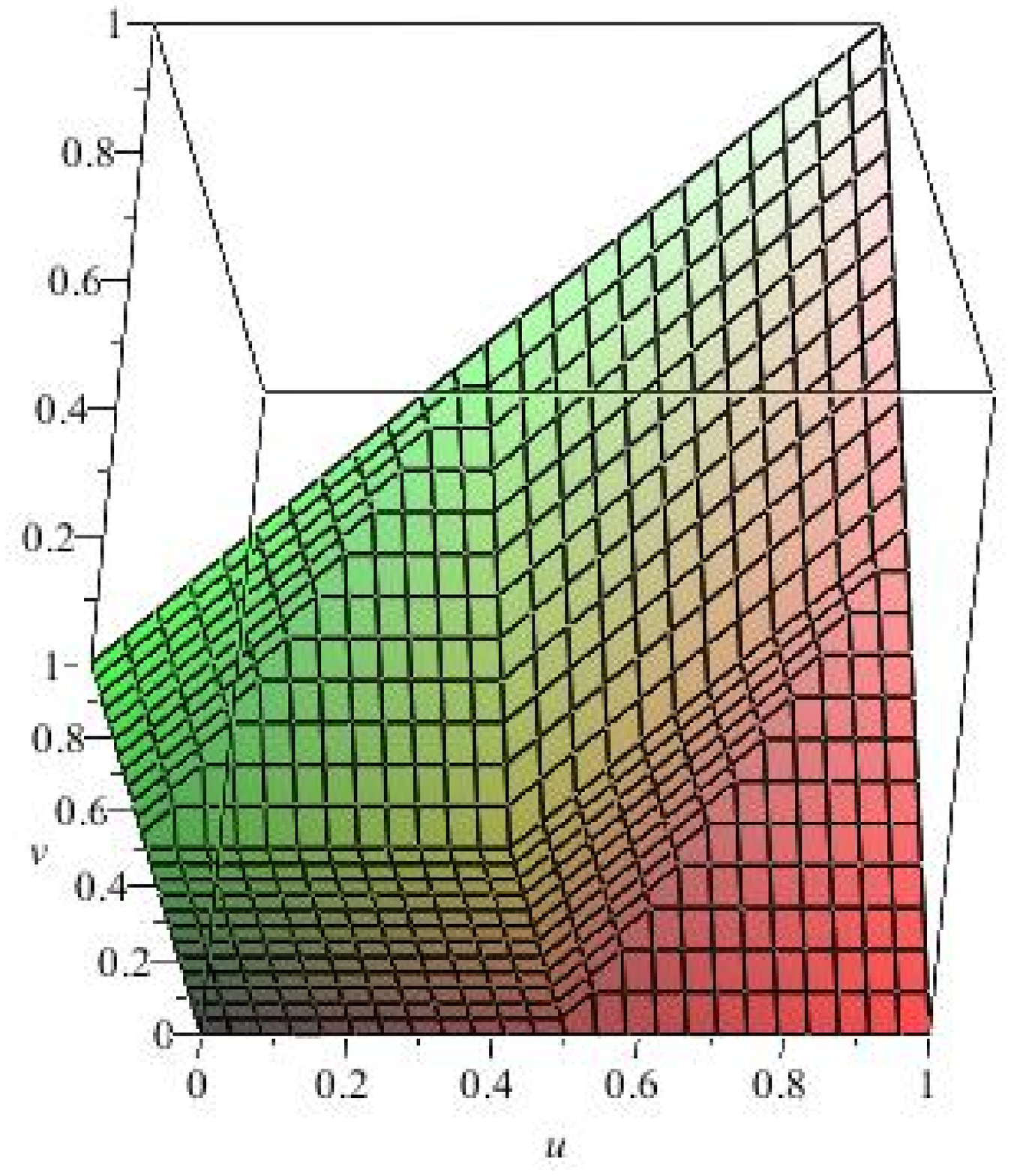}
  \hspace{1mm}
  \includegraphics[width=0.3\textwidth,keepaspectratio]{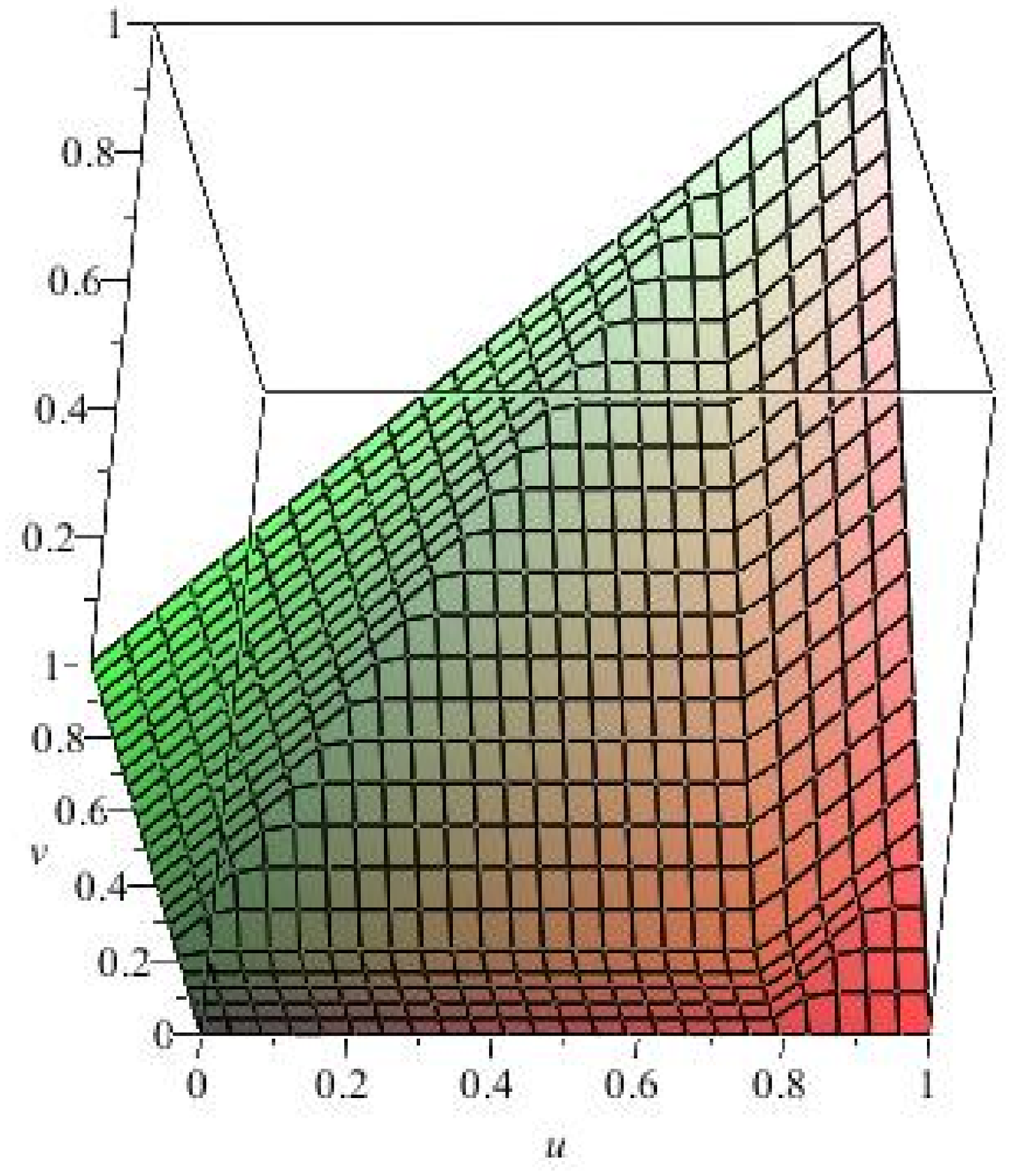}
  \caption{\small{Copulas from the family $C_\gamma$ with $\gamma = 0.3,\, 0.5,\, 0.8$ respectively}}
  \label{fig:Cgamma}
\end{figure}

\section{Further properties of $\mathcal{L}_\gamma$ and examples}\label{sec3}

We start this Section by analyzing further properties of the classes
$\mathcal{L}_{\gamma}$ that can also shed light on the relations between
stochastic precedence and stochastic orderings.
First we notice that the previous Definition \ref{def2} has been formulated
in terms of the usual stochastic ordering $\preceq_{st}$. However
similar results can also be obtained for other important
concepts of stochastic ordering that have been considered
in the literature (such as the \textit{hazard rate}, the
\textit{likelihood ratio}, and the \textit{mean residual life}
orderings, see \cite{ShSh}).

Let us fix, in fact, a stochastic ordering $\preceq_{\ast}$ different
from $\preceq_{st}$. Definition \ref{def2} can be modified by
replacing therein $\preceq_{st}$ with $\preceq_{\ast}$ and this
operation leads us to a new class of copulas that we can denote by
$\mathcal{L}_{\gamma}^{(\ast)}$. More precisely we set
\begin{equation}\label{LgammaCstar}
\mathcal{L}^{(*)}_\gamma:=\{C\in\mathcal{C}:\eta(C,G_1,G_2)\geq\gamma,\;
\forall \,G_1,G_2\in\mathcal{G}\; s.t. \;G_1\preceq_{*}G_2\}
\end{equation}
or equivalently
\begin{equation}\label{LgammaCbis}
\mathcal{L}^{(*)}_\gamma =\{C\in\mathcal{C}:\eta^*(C)\geq\gamma\}
\end{equation}
where
\begin{equation}\label{Lgamma2bis}
\eta^*(C):=\inf_{G_1,G_2\in\mathcal{G}}\{\eta(C,G_1,G_2):G_1\preceq_{\ast}G_2\}.
\end{equation}
For given $\gamma\in(0,1)$, one
might wonder about possible relations between
$\mathcal{L}_{\gamma}^{(\ast)}$ and $\mathcal{L}_{\gamma}$. Actually
one has the following result, which will be formulated for binary relations (not necessarily
stochastic orderings) over the space $D(\mR)$.

\begin{proposition} \label{altri}
Let $\preceq_{\ast}$ be a relation satisfying
\begin{itemize}
    \item[(a)] for any $G\in D(\mathbb{R})$ one has $G\preceq_{\ast} G$;
    \item[(b)] for any $G_1,G_2\in D(\mathbb{R})$ with $G_1\preceq_{\ast}G_2$ one has $G_1\preceq_{st}G_2 $.
\end{itemize}
Then $\mathcal{L}_{\gamma}=\mathcal{L}_{\gamma}^{(\ast)}$.
\end{proposition}

\begin{proof}
In view of (b), one has that $\eta(C) \leq \eta^*(C)$. In fact both the quantities $\eta(C)$ and $\eta^*(C)$
are obtained as an infimum of the same functional and, compared with $\eta$, the quantity $\eta^*$ is an
infimum computed on a smaller set.

Due to (a), however, $\eta(C)$ and $\eta^*(C)$ are both obtained, in \eqref{Lgamma2} and \eqref{Lgamma2bis}
respectively, as minima attained on a same point $(G,G)$. We can then conclude that
$\mathcal{L}^{(*)}_\gamma =\mathcal{L}_\gamma$.
\end{proof}

Concerning Proposition \ref{altri} we notice that, for example, the hazard rate and the likelihood ratio orderings,
$\preceq_{hr}$ and $\preceq_{lr}$, both satisfy the conditions (a) and (b).

In applied problems it can be relevant to remark that imposing stochastic orderings stronger than $\preceq_{st}$
does not necessarily increase the level of stochastic precedence.

\bigskip

For the sake of notational simplicity we come back to considering the usual
stochastic ordering $\preceq_{st}$ and the class $\mathcal{L}_\gamma$.

For what follows it is now convenient also to consider the quantities $\xi(C,G_1,G_2)$ and $\xi(C)$ defined
as follows:
\begin{align}\label{xi}
\xi(C,G_1,G_2)&:=\mathbb{P}(X_1=X_2),\\ \label{xiC}
\xi(C)&:=\xi(C,G,G),
\end{align}
where $X_1$ and $X_2$ are random variables with distributions $G_1,G_2\in\mathcal{G}$
respectively and connecting copula $C$.

For a given bivariate model we have considered so far the quantities $\eta(C)$
with $C$ denoting the connecting copula. In what follows we point out the relations among
$\eta(C)$, $\eta(\widehat{C})$, $\eta(C^t)$ where $\widehat{C}$ and $C^t$
denote the \emph{survival copula} and the \emph{transposed copula}, respectively.
The transposed copula $C^t$ is defined by
\begin{equation}\label{Ctransposed}
C^t(u,v):=C(v,u)
\end{equation}
so that if $C$ is the connecting copula of the pair $(X_1,X_2)$, then $C^t$ is the
copula of the pair $(X_2,X_1)$. Whence, if $X_1$ and $X_2$ have the same distribution $G\in\mathcal{G}$, then
$$
\eta(C^t)=\mP(X_2\leq X_1).
$$

On the other hand the notion of survival copula of the pair $(X_1,X_2)$, which comes out as natural when
considering pairs of non-negative random variables, is defined by the equation
\begin{equation}\label{JointSurvivFnctConC}
\overline{F}_{X_1,X_2}(x_1,x_2)=\widehat{C}\left[
\overline{G}_1(x_1),\overline{G}_2(x_2)\right],
\end{equation}
with $\overline{G}_1$ and $\overline{G}_2$ respectively denoting the
marginal survival functions:
$$
\overline{G}_1(x_1)=\mathbb{P}(X_1>x_1),\quad\overline{G}_2(x_2)=\mathbb{P}(X_2>x_2)\,.
$$

The relationship between the survival copula $\widehat{C}$ of $(X_1,X_2)$ and the
connecting copula $C$ is given by (see \cite{nelsen2006introduction})
\begin{equation}\label{Csurvival}
\widehat{C}(u,v)=u+v-1+C(1-u,1-v).
\end{equation}

The following result shows the relations tying the different quantities
$\eta(C)$, $\eta(\widehat{C})$, $\eta(C^{t})$.
The proof is easy and will be omitted.

\begin{proposition}\label{tuttiglieta}
Let $C\in\mathcal{C}$. The following relation holds:
$$
\eta(\widehat{C})=\eta(C^t)=1-\eta(C)+\xi(C).
$$
\end{proposition}

A basic property of the classes $\mathcal{L}_{\gamma}$ and $\mathcal{B}_{\gamma}$ is given by the following result.

\begin{proposition}\label{convex}
\label{conv} For $\gamma\in[0,1]$, the classes
$\mathcal{L}_{\gamma}$, $\mathcal{L}_{\gamma}^{c}=\mathcal{C}\setminus\mathcal{L}_{\gamma}$, and
$\mathcal{B}_{\gamma}$ are convex.
\end{proposition}

\begin{proof}
We consider two bivariate copulas $C_1,C_2\in\mathcal{L}_{\gamma}$ and a convex
combination of them, i.e. take $\alpha\in(0,1)$ and $C:=\alpha C_1+(1-\alpha)C_2$.
We show that $C\in\mathcal{L}_{\gamma}$, indeed
$$
\eta(C)=\int_A dC(u,v)=\alpha\int_A dC_1(u,v)+(1-\alpha)\int_A dC_2(u,v)
=\alpha\eta(C_1)+(1-\alpha)\eta(C_2).
$$
Since $\eta(C_1),\eta(C_2)$ are larger or equal than $\gamma$ then
$\eta(C)\geq\gamma$, whence $\mathcal{L}_{\gamma}$ is convex. Now one can use the
same argument in order to show that $\mathcal{L}_{\gamma}^c$ and
$\mathcal{B}_{\gamma}$ are convex as well.
\end{proof}

An immediate application of Proposition \ref{convex} concerns the case when, given a random parameter $\Theta$,
all the connecting copulas of the conditional distributions of $(T,X)$, belong to a same class $\mathcal{L}_\gamma$.
Proposition \ref{convex} in fact, guarantees that the copula of $(T,X)$ belongs to $\mathcal{L}_\gamma$ as well.

Some aspects of the definitions and results given so far will be demonstrated here by presenting a few examples.
We notice that, as shown by Proposition \ref{BCgamma}, the condition $\preceq_{st}$ does not imply $\preceq_{sp}^{(\gamma)}$,
with $\gamma\in(0,1)$.
For the special case $\gamma=1/2$ we now present an example of applied interest.

\begin{example}\label{ex:1}
\end{example}
Let $X,Y$ be two non-negative random variables, where $Y$ has an exponentially density $f_Y(y)$
with failure rate $\lambda$ and where stochastic dependence between $X$ and $Y$ is described by
a ``load-sharing'' dynamic model as follows: conditionally on $(Y=y)$, the failure rate of
$X$ amounts to $\alpha=1$ for $t<y$ and to $\beta$ for $t>y$. We assume $1<\lambda<\beta<1+\lambda$.
This position gives rise to a jointly absolutely continuous distribution for which
we can consider
$$
\mathbb{P}(X>x|Y=y) := \int_x^{+\infty}f_{X,Y}(t,y)dt,
$$
$f_{X,Y}$ denoting the joint density of $X,Y$.
As to the survival function of $X$, for any fixed value $x>0$, we can argue as follows.
\begin{align*}
\overline{F}_X(x)&:=\mathbb{P}(X>x)=\int_0^{+\infty}\mathbb{P}(X>x|Y=y)f_Y(y)dy\\
&=\int_0^x\mathbb{P}(X>x|Y=y)f_Y(y)dy + \int_x^{+\infty}\mathbb{P}(X>x|Y=y)f_Y(y)dy\\
&=\int_0^x\mathbb{P}(X>y|Y=y)\mathbb{P}(X>x|Y=y,X>y)f_Y(y)dy + \int_x^{+\infty}e^{-x}f_Y(y)dy\\
&=\left(1-\frac{\lambda}{1+\lambda-\beta}\right)e^{-(1+\lambda)x} +
\frac{\lambda}{1+\lambda-\beta}\,e^{-\beta x}\leq e^{-\lambda x}, \quad\forall x>0.
\end{align*}
We can then conclude that $X\preceq_{st}Y$. On the other hand the same position gives also rise to
$\mathbb{P}(X\leq Y)=1/(1+\lambda)<1/2$.

\medskip

The next example shows that for three random variables $T,X',X''$, the implication
$T\preceq_{st}X^{\prime}\preceq_{st}X^{\prime\prime}\Rightarrow\mathbb{P}(T\leq X^{\prime\prime})\leq\mathbb{P}(T\leq X^{\prime})$
can fail when the connecting copulas of $(T,X')$ and $(T,X'')$ are different.

\begin{example}\label{ex:2}
\end{example}
Let $Y_1,\ldots,Y_5$ be i.i.d. random variables, with a continuous
distribution and defined on a same probability space, and set
\[
T=\min(Y_1,Y_2), \;X^{\prime}=\max(Y_1,Y_2), \;X^{\prime\prime}=\max(Y_3,Y_4,Y_5).
\]
Thus $X^{\prime}\preceq_{st}X^{\prime\prime}$, but $\mathbb{P}(T\leq X^{\prime})=1$ and
$\mathbb{P}(T\leq X^{\prime\prime})<1$.

\bigskip

\begin{remark}\label{rem:contidiretti}
For some special types of copula $C$, the computation of $\eta(C,G_1,G_2)$ can be carried out directly,
in terms of probabilistic arguments, provided the distributions $G_1,G_2$ belong to some appropriate class.
This circumstance in particular manifests for the models considered in the subsequent examples.
Let $C$ be a copula satisfying such conditions. Then Proposition \ref{eta_ineq} can be used to obtain
inequalities for $\eta(C,H_1,H_2)$ even if $H_1,H_2$ do not belong to $\mathcal{G}$ provided, e.g.,
that $H_1\preceq_{st}G_1$, $G_2\preceq_{st}H_2$ and $G_1,G_2\in\mathcal{G}$.
\end{remark}

The next example will be devoted to bivariate gaussian models, i.e. to a relevant case of symmetric copulas.


\begin{example}
\textbf{Gaussian Copulas.}
\end{example}

The family of bivariate gaussian copulas (see e.g. \cite{nelsen2006introduction})
is parameterized by the correlation coefficient
$\rho\in(-1,1)$. The corresponding copula $C^{(\rho)}$ is absolutely continuous and symmetric, and
$\eta(C^{(\rho)})=1/2$ and, thus, it does not depend on $\rho$.
For fixed pairs of distributions $G_1,G_2$, on the contrary, the quantity $\eta(C^{(\rho)},G_1,G_2)$
does actually depend on $\rho$, besides on $G_1$ and $G_2$.
This class provides the most direct instance of the situation outlined in the above Remark
\ref{rem:contidiretti}.
The value for $\eta(C^{(\rho)},G_1,G_2)$ is in fact immediately obtained when $G_1,G_2$ are gaussian.
Let $X_1,X_2$ denote gaussian random variables with connecting copula $C^{(\rho)}$ and parameters
$\mu_1,\mu_2,\sigma_1^2,\sigma_2^2$. Since the random variable $Z=X_1-X_2$ is distributed according to $\mathcal{N}(\mu_1-\mu_2,\sigma^2_1+\sigma^2_2-2\rho\sigma_1\sigma_2)$ we can write
\begin{equation}\label{EtaGaussianoSpec}
\eta(C^{(\rho)},G_1,G_2)=
\mP(Z\leq0)=\Phi\left(\frac{\mu_2-\mu_1}{\sqrt{\sigma^2_1+\sigma^2_2-2\rho\sigma_1\sigma_2}}\right).
\end{equation}

We recall that, when $X_i\sim \mathcal{N}(\mu_i,\sigma_i^2)$ for $i=1,2$, the necessary and sufficient condition for
$X_1\preceq_{st}X_2$ is $\mu_1\leq\mu_2$ and $\sigma_1=\sigma_2$ (see e.g. \cite{AKS}).
In other words, for $G_1,G_2$ gaussian, $G_1\preceq_{st}G_2$ means $X_1\preceq_{sp}X_2$ and $\sigma_1=\sigma_2$.
By using the formula in \eqref{EtaGaussianoSpec}, with $\sigma_1=\sigma_2=\sigma$, we have
\begin{equation}\label{EtaGaussiano}
\eta(C^{(\rho)},G_1,G_2)=\Phi\left(\frac{\mu_2-\mu_1}{\sigma\sqrt{2(1-\rho)}}\right).
\end{equation}
Thus $G_1\preceq_{st}G_2\Rightarrow\eta(C^{(\rho)},G_1,G_2)\geq 1/2$, as shown by Proposition \ref{etainf} and Theorem \ref{continuo}.
We notice that $\eta(C^{(\rho)},G_1,G_2)$ is an increasing function of $\rho$.

Proposition \ref{eta_ineq} can be extended to obtain, say, that
$$
\eta(C^{(\rho)},G_1,G_2)\leq\eta(C^{(\rho)},H_1,H_2),
$$
when $H_1\preceq_{st}G_1$ and $G_2\preceq_{st}H_2$, for $G_1,G_2\in\mathcal{G}$ and $H_1,H_2\notin\mathcal{G}$.
We then can give inequalities for $\eta(C^{(\rho)},H_1,H_2)$
in terms of \eqref{EtaGaussianoSpec}, provided $H_1,H_2$ are suitably comparable in the
$\preceq_{st}$ sense with gaussian distributions.

\bigskip

In the cases when $\xi(C)>0$, we should obviously distinguish between computations of
$\mathbb{P}(X_1\leq X_2)$ and $\mathbb{P}(X_1<X_2)$, where $C$ is the connecting copula of $X_1,X_2$.
A remarkable case when this circumstance happens is considered in the following example.

\bigskip

\begin{example}
\textbf{Marshall-Olkin Models}
\end{example}

We consider the Marshall-Olkin copulas (see e.g \cite{joe1997,MS1987,nelsen2006introduction}),
namely those whose expression is the following:
$$
\widehat{C}^{(\alpha_1,\alpha_2)}(u,v):=u\,v\,\min\{u^{-\alpha_1},\,v^{-\alpha_2}\}
$$
for $0<\alpha_i<1$, $i=1,2$.
We notice that the Marshall-Olkin copula has a singular part that is concentrated on the curve $u^{\alpha_1}=v^{\alpha_2}$
(see also Figure \ref{fig:MO}). 
Actually the measure of such a singular component is given by
$$
\frac{\alpha_1\alpha_2}{\alpha_1+\alpha_2-\alpha_1\alpha_2}.
$$

\begin{figure}[ht]
  \includegraphics[width=0.3\textwidth,keepaspectratio]{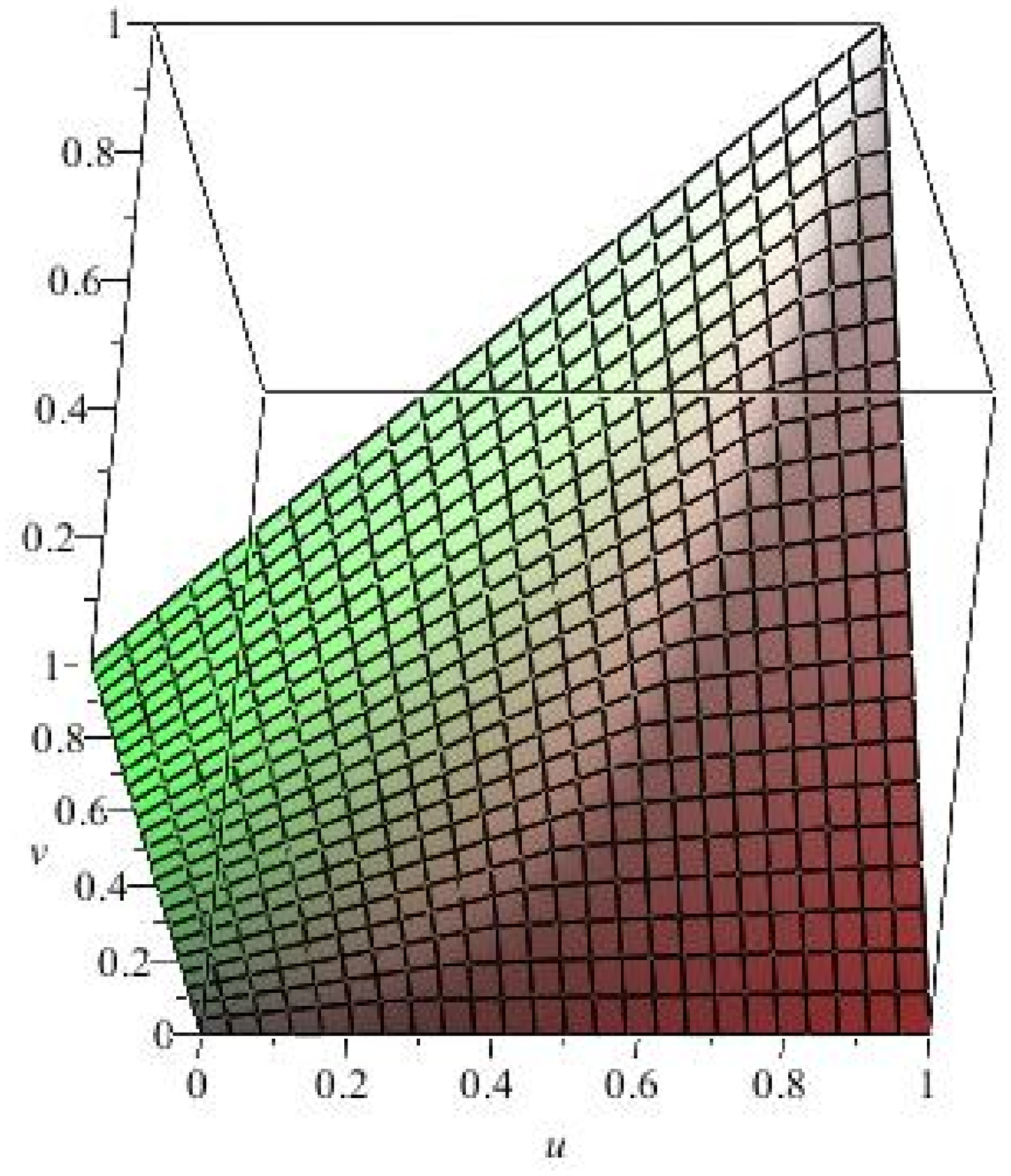}
  \qquad  \qquad
  \includegraphics[width=0.27\textwidth,keepaspectratio]{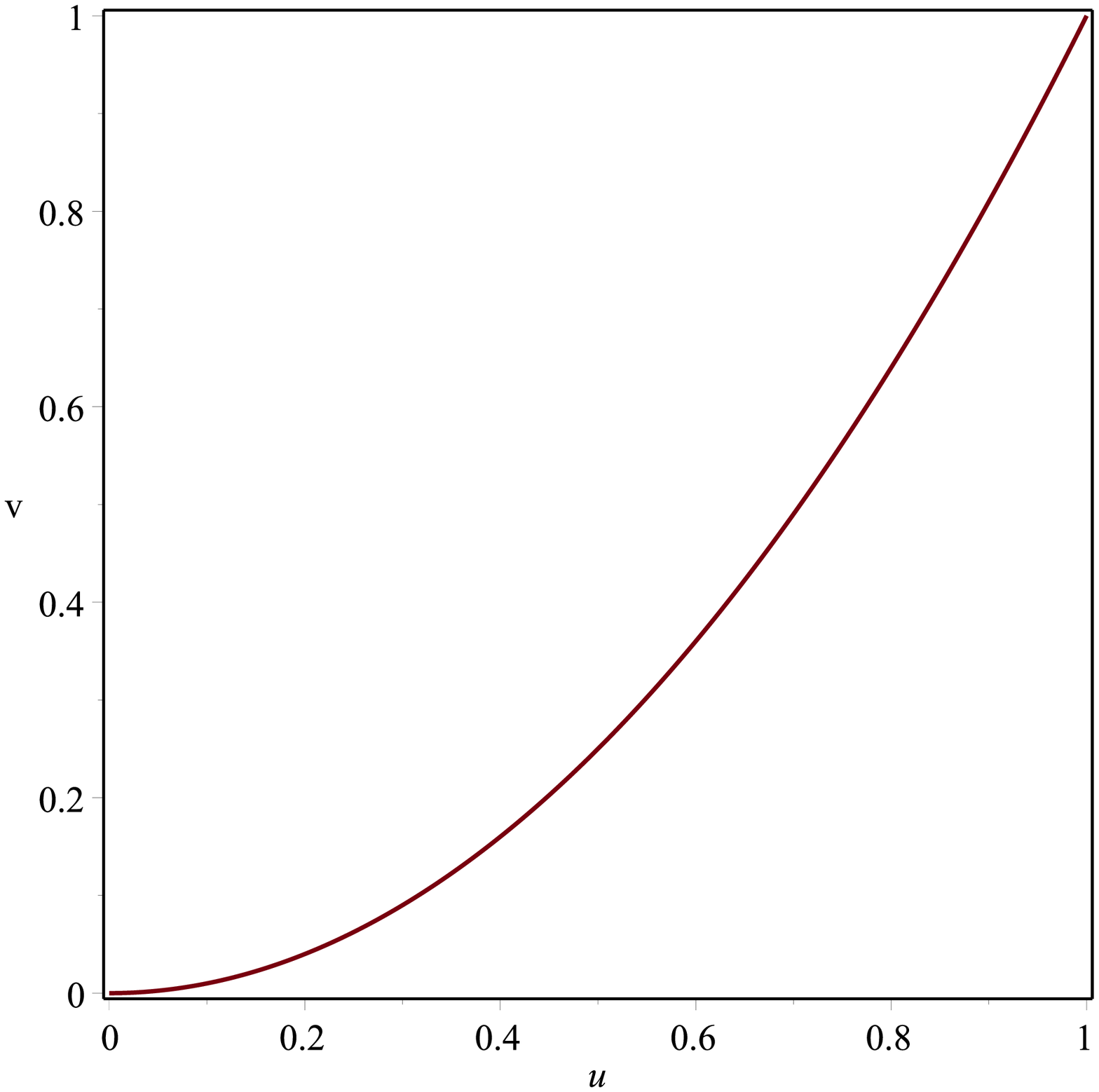}
  \caption{\small{Marshall-Olkin Copula (left) and graph of $u^{\alpha_1}=v^{\alpha_2}$ (right). Special case $\alpha_1=0.4,\,\alpha_2=0.2$.}}
  \label{fig:MO}
\end{figure}

As for the computation of $\eta(\widehat{C}^{(\alpha_1,\alpha_2)})$ we use the expression in \eqref{etapuro}.
By separately considering the curve $u^{\alpha_1}=v^{\alpha_2}$ and the domains where $\widehat{C}^{(\alpha_1,\alpha_2)}$ is absolutely continuous, we obtain
$$
\eta(\widehat{C}^{(\alpha_1,\alpha_2)})=\frac 1{2-\alpha_1\wedge\alpha_2}
\left(1-\frac{(\alpha_1-\alpha_1\wedge\alpha_2)(\alpha_1\wedge\alpha_2)}{\alpha_1-\alpha_2}\right).
$$
Consider the copula
$$
C^{(\alpha_1,\alpha_2)}(u,v):=\widehat{C}^{(\alpha_1,\alpha_2)}(1-u,1-v)+u+v-1.
$$
We will see now that the value of $\eta(C^{(\alpha_1,\alpha_2)},G_1,G_2)$ directly follows from probabilistic arguments,
provided $G_1,G_2$ are exponential distributions with appropriate parameters.
Let in fact $V$, $W$ and $Z$ be three random variables independent and exponentially distributed with
parameters $\mu_1=1/\alpha_1-1$, $\mu_2=1/\alpha_2-1$ and $\mu=1$, respectively. The new random variables
$$
X_1:=V\wedge Z,\qquad X_2:=W\wedge Z,
$$
have survival copula $\widehat{C}^{(\alpha_1,\alpha_2)}$, connecting copula $C^{(\alpha_1,\alpha_2)}$, and exponential distributions $G_1^{(\alpha_1)}$ and $G_2^{(\alpha_2)}$, with parameters $1/\alpha_1$ and $1/\alpha_2$ respectively.
We now proceed with the computation of
$$
\eta(C^{(\alpha_1,\alpha_2)},G_1^{(\alpha_1)},G_2^{(\alpha_2)})=\mP(X_1\leq X_2).
$$
We can write
\begin{align*}
\xi(C^{(\alpha_1,\alpha_2)},G_1^{(\alpha_1)},G_2^{(\alpha_2)})&=\mP(X_1=X_2)=\mP(Z\leq V\wedge W)\\
&=\frac{1}{\mu_1+\mu_2+1}=\frac{\alpha_1\alpha_2}{\alpha_1+\alpha_2-\alpha_1\alpha_2},
\end{align*}
furthermore
\begin{equation*}
\mP(X_1<X_2)=\mP(V<W\wedge Z)=\frac{(1-\alpha_1)\,\alpha_2}{\alpha_1+\alpha_2-\alpha_1\alpha_2},
\end{equation*}
and finally we obtain
$$
\mP(X_1\leq X_2)=\frac{\alpha_2}{\alpha_1+\alpha_2-\alpha_1\alpha_2}.
$$
Then
$$
\eta(C^{(\alpha_1,\alpha_2)},G_1^{(\alpha_1)},G_2^{(\alpha_2)})
=\frac{\alpha_2}{\alpha_1+\alpha_2-\alpha_1\alpha_2}.
$$

We now conclude this Section with an example showing an extreme case in the direction of Remark \ref{rem:contidiretti}.

\begin{example}
\textbf{Copulas of order statistics.}
\end{example}

Let $A,B$ be two i.i.d. random variables with distribution function
$G\in\mathcal{G}$ and denote by $X_1,X_2$ their order statistics, namely
$X_1=\min\{A,B\},X_2=\max\{A,B\}$.
The distributions of $X_1,X_2$ depend on $G$ and are respectively given by
\begin{align*}
F_1^{(G)}(x_1)&=\mP(\min\{X_1,X_2\}\leq x_1)=2G(x_1)-G(x_1)^2,\\
F_2^{(G)}(x_2)&=\mP(\max\{X_1,X_2\}\leq x_2)=G(x_2)^2\,.
\end{align*}
The connecting copula of $(X_1,X_2)$, represented in Figure \ref{fig:stat_ord}, is given by
$$
K(u,v)=\left\{
\begin{array}{ll}
\displaystyle{2(1-(1-u)^{1/2})v^{1/2}-(1-(1-u)^{1/2})^{2}} & \mbox{ if } v\geq(1-(1-u)^{1/2})^2,\\
\displaystyle{v} & \mbox{ otherwise}.
\end{array}
\right.
$$
We have, by definition,
$$
\eta(K,F_1^{(G)},F_2^{(G)})=1,
$$
and it does not depend on $G$. We notice, on the other hand, that the computation of $\eta(K)=\eta(K,G,G)$,
with $G\in\mathcal{G}$, is to be carried out explicitly, since the pair $(G,G)$ can never appear as the
pair of marginal distributions of order statistics. By recalling \eqref{formulautile} one obtains
$$
\eta(K)=\int_{[0,1]^2}\frac{\II_A(u,v)}{2\sqrt{v}\sqrt{1-u}}\,dv\,du=2-\frac{\pi}{2}<\frac 12.
$$
We can extend this example to the case when the connecting copula of $A,B$ is a copula $D$ different from the product
copula $\Pi$, but still $A$ and $B$ are identically distributed according to a distribution function $G$.
In this case the connecting copula $K$ of $X_1,X_2$ depends on $D$, but again it does not depend on $G$
(see \cite{navarro2010relationships} page 478).

\begin{figure}[ht]
  \includegraphics[width=0.3\textwidth,keepaspectratio]{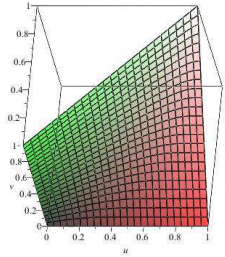}
    \caption{\small{Ordered Statistic Copula K}}
    \label{fig:stat_ord}
\end{figure}

\section{The Target-Based Approach to decisions under risk and the classes $\mathcal{L}_{\gamma}$.}\label{sec4}

In this Section we will look at the arguments of the previous Sections in the
perspective of one-attribute decisions problems under risk and, more in
particular, of the related \textit{Target-Based Approach} (\textit{TBA}).
In such problems, a \textit{risky prospect} (or \textit{lottery}) $X$ is
nothing else than a real random variable representing, say, the random amount
of wealth obtained as the consequence of an action or of an economic
investment. An investor (or decision-maker) $I$ is supposed to choose one out
of many different actions by evaluating and comparing the different
probability distributions corresponding to any single prospect. This choice
is implemented on the basis of $I$'s attitudes toward risk.

As very well-known, in such a frame, the \textit{Expected Utility Principle}
first suggests that $I$ describe her/his own attitudes by means of a utility
function $U$ ($U:\mathbb{R}\rightarrow\mathbb{R}$) and consequently prescribes
that any prospect $X$ (with its probability distribution $F_X$) be evaluated
in terms of the expected-utility
\[
\mathbb{E}(U(X))=\int_\mathbb{R} U(x)F_X(dx).
\]

In the same frame, the \textit{Target-Based Approach} is based on a different principle.
The TBA assumes in fact that the exclusive
interest of the investor $I$, in the use of the amount of wealth $X$, is concentrated on the possibility
of ``buying'' a specific good (a house, a car, a block of shares of a stock, etc.). The price of such good is a random variable
$T$ (the \textit{target}), with a probability distribution $F_T$.
Whence $I$ is for first supposed to specify the target $T$ as a way to describe
his/her own attitude with respect to risk. Then $I$ will evaluate any single prospect $X$ in terms of
the probability $\mP(T\leq X)$. The best prospect will be the one that maximizes $\mP(T\leq X)$.
Such an approach was proposed by Bordley, Li Calzi,
and Castagnoli (see \cite{BorLic00,CasLic96}). Some related ideas
were already around in the economic literature in the past and other
interesting developments appeared in the subsequent years, especially for what
concerns the multi-attribute setting. Generally $T$ and $X$ may in fact also be vectors.

Here we concentrate
attention on the single-attribute case where $(T,X)$ are pairs
of real-valued random variables. It is clear then that the objects of central
interest in the TBA are, for a fixed target $T$, the probabilities
$\mathbb{P}(T\leq X)$ and the analysis developed in the
previous sections can reveal of interest.
We assume the existence of regular conditional distributions. In particular we
assume that, for any prospect $X$, we can determine
$\upsilon_T^{(X)}(x):=\mathbb{P}(T\leq x|X=x)$, so that we can write
\begin{equation*}
\mathbb{P}(T\leq X)=\int_\mathbb{R}\upsilon_T^{(X)}(x)\,dF_X(x).
\end{equation*}

Before continuing it is useful to look at the special case when $X$ and $T$ are stochastically independent.
We can thus write
\begin{align*}
\mathbb{P}(T\leq X) &  =\int_{\mathbb{R}}\upsilon_T^{(X)}(x)\,dF_X(x)
=\int_\mathbb{R}F_T(x)\,dF_X(x).
\end{align*}

We notice that, in such a case, $\mathbb{P}(T\leq X)$ can be seen as the
expected value of a utility: by considering $U=F_T$ as the utility function,
we have
\[
\mathbb{E}(U(X))=\int_\mathbb{R}U(x) \,dF_X(x)=\int_\mathbb{R} F_T(x) \,dF_X(x)=\mathbb{P}(T\leq X).
\]
Under the condition of independence, any bounded and right-continuous utility function can thus be seen as the
distribution function of a target $T$, and vice-versa. Such an approach gives
rise to easily-understandable and practically useful interpretations of
several notions of utility theory.
TBA however becomes, in a sense, more general than the expected utility
approach by allowing for stochastic dependence between targets and prospects.
In fact the TBA considers more general decision rules, if we admit the
possibility of some correlation between the target and the prospects. If
$X$ and $T$ are not independent, $\upsilon_T^{(X)}(x)$ does not
coincide anymore with the distribution function $F_T(x)$.
For further discussion see again \cite{BorLic00,CasLic96}.

We now briefly summarize the arguments of Sections \ref{sec2} and \ref{sec3} in the perspective
of a decision problem where, for a fixed target $T$, we aim to rank two
different prospects with marginal distributions $G_{X_1},G_{X_2}$ and with
connecting copulas $C_1,C_2$, corresponding to the pairs
$(T,X_1)$ and $(T,X_2)$, respectively.

In the case of independence, a prospect $X_2$ should be obviously preferred
to a prospect $X_1$ if $X_1\preceq_{st}X_2$. In the case of
dependence, on the contrary, this comparison is not sufficient anymore. In
fact the choice of a prospect $X$ should be based not only on the
corresponding distribution $F_X$, but also on the connecting copula of the
pair $(T,X)$.

For fixed $C$, the quantity $\eta(C,G_T,G_X)=\mP(T\leq X)$
is equal to the quantity $\eta(C)$ for all pairs such that $G_T=G_X=G$
with $G$ belonging to $\mathcal{G}$ (See Proposition ~\ref{foranypair}). For $G_T\neq G_X$,
the implication $T\preceq_{st}X\Rightarrow \mP(T\leq X)\geq \gamma$
does not necessarily hold (see Proposition \ref{BCgamma} Example \ref{ex:1}). For two different prospects $X_1,X_2$,
Proposition ~\ref{eta_ineq} guarantees that, when
$C_1=C_2=C$, the condition $G_T\preceq_{st}G_{X_1}\preceq_{st}G_{X_2}$ implies
$\eta(C,G_T,G_{X_1})=\mP(T\leq X_1)\leq\eta(C,G_T,G_{X_2})=\mP(T\leq X_2)$.
As shown by Example ~\ref{ex:2}, when $C_1\neq C_2$, we can have both the conditions
$\eta(C_1,G_T,G_{X_1})>\eta(C_2,G_T,G_{X_2})$ and $G_T\preceq_{st}G_{X_1}\preceq_{st}G_{X_2}$ $(G_{X_1}\neq G_{X_2})$.
Concerning the quantities $\eta(C_1,G_T,G_{X_1})$ and $\eta(C_2,G_T,G_{X_2})$,
Theorems \ref{Lgamma} and \ref{continuo} show that, for $G_T\preceq_{st}G_{X_i}$ ($i=1,2$),
$$
\mP(T\leq X_i) = \eta(C_i,G_T,G_{X_i})\geq\eta(C_i).
$$

Let us consider the case when the only available information about $C_1$ and $C_2$
is that $\eta(C_i)\geq\gamma_i$ (i.e. that $C_i$ belongs to the class
$\mathcal{L}_{\gamma_i}$). Then a rough and conservative choice between
$X_1$ and $X_2$ suggests to select $X_i$ with the larger value of
$\gamma_i$, provided $G_{X_1}\preceq_{st}G_{X_2}$ or that
$X_1,X_2$ are nearly identically distributed.

\bigskip


\bibliographystyle{abbrv}

\end{document}